\documentclass[11pt]{amsart}

\usepackage{amsfonts}

\usepackage{hyperref}

\allowdisplaybreaks

\newcommand{\N}{\mathbb{N}}
\newcommand{\Z}{\mathbb{Z}}
\newcommand{\R}{\mathbb{R}}
\newcommand{\T}{\mathbb{T}}

\newcommand{\dis}{\mathrm{d}}

\theoremstyle{plain}

\newtheorem{prop}{Proposition}

\newtheorem{thm}{Theorem}
\newtheorem{coro}{Corollary}

\theoremstyle{remark}
\newtheorem{rem}{Remark}

\title[Harmonic Analysis associated with a discrete Laplacian]{Harmonic Analysis associated with a discrete Laplacian}

\author[\'O. Ciaurri, T. A. Gillespie, L. Roncal, J. L. Torrea, and J. L. Varona]{\'Oscar Ciaurri \and T. Alastair Gillespie \and Luz Roncal \and Jos\'e L. Torrea \and Juan Luis Varona}

\address[\'O. Ciaurri, L. Roncal, and J. L. Varona]{%
    Departamento de Matem\'aticas y Computaci\'on\\
    Universidad de La Rioja\\
    26004 Logro\~no, Spain}
\email{\{oscar.ciaurri,luz.roncal,jvarona\}@unirioja.es}

\address[T. A. Gillespie]{School of Mathematics and Maxwell Institute for Mathematical Sciences\\
    University of Edinburgh\\ 
    Edinburgh EH9 3JZ, Scotland, U.K.}
\email{t.a.gillespie@ed.ac.uk}

\address[J. L. Torrea]{Departamento de Matem\'aticas, Facultad de Ciencias\\
	Universidad Aut\'onoma de Madrid\\
	28049 Madrid, Spain\\
and
Instituto de Ciencias Matem\'aticas (CSIC-UAM-UC3M-UCM)}
\email{joseluis.torrea@uam.es}

\thanks{Research partially supported by grants MTM2012-36732-C03-02 and MTM2011-28149-C02-01 from Spanish Government}



\begin{document}

\begin{abstract}

It is well-known that the fundamental solution of
$$
u_t(n,t)= u(n+1,t)-2u(n,t)+u(n-1,t), \quad n\in\mathbb{Z},
$$
with  $u(n,0) =\delta_{nm}$ for every fixed $m \in\mathbb{Z}$,
is given by $u(n,t) = e^{-2t}I_{n-m}(2t)$, where $I_k(t)$ is the Bessel function of imaginary argument.
In other words, the heat semigroup of the discrete Laplacian is described  by the formal series
$$
W_tf(n) = \sum_{m\in\mathbb{Z}} e^{-2t} I_{n-m}(2t) f(m).
$$
By using  semigroup theory, this formula allows us to analyze some operators associated with the discrete  Laplacian. In particular, we obtain the maximum principle for the discrete fractional Laplacian, weighted $\ell^p(\mathbb{Z})$-boundedness  of conjugate harmonic functions, Riesz transforms and square functions of Littlewood-Paley.

Interestingly,  it is shown that the Riesz transforms coincide essentially with the so called discrete Hilbert transform defined by D.~Hilbert at the beginning of the XX century. We also see that these Riesz transforms are limits of the conjugate harmonic functions.

The results rely on a careful use of several properties of Bessel functions.
\end{abstract}

\maketitle

\section{Introduction}

The purpose of this paper is the analysis of several operators associated with the discrete Laplacian
\[
\Delta_{\dis} f(n) = f(n+1)-2f(n)+f(n-1),\qquad n \in \Z,
\]
in a similar way to the analysis of classical operators associated with the Euclidean Laplacian $\Delta = \frac{\partial^2}{\partial x^2}$. Among them, we study the (discrete) fractional Laplacian, maximal heat and Poisson semigroups, square functions, Riesz transforms and conjugate harmonic functions.

In order to define and study these operators we shall use the heat semigroup $W_t=e^{t\Delta_{\dis}}$ as a fundamental tool. This is not difficult because the fundamental solution of
\[
u_t(n,t)= u(n+1,t)-2u(n,t)+u(n-1,t), \quad n\in\mathbb{Z},
\]
with $u(n,0) =\delta_{nm}$ for every fixed $m \in\mathbb{Z}$,
is given by $u(n,t) = e^{-2t}I_{n-m}(2t)$ (see \cite{Gr} and \cite{GrIl}), where $I_k(t)$ is the Bessel function of imaginary argument (for these functions, see \cite[Chapter~5]{Lebedev}). Consequently, the heat semigroup is given by the formal series
\begin{equation}\label{primercalor}
  W_tf(n) = \sum_{m\in\Z} e^{-2t}I_{n-m}(2t)f(m).
\end{equation}
Then, the function $u(n,t) = W_tf(n)$ defined in~\eqref{primercalor} is the solution to the \textit{discrete} heat equation
\begin{equation}
\label{eq:calorZ}
\begin{cases}
\frac{\partial}{\partial t} u(n,t) =  u(n+1,t) - 2u(n,t) + u(n-1,t), \\[4pt]
u(n,0) = f(n),
\end{cases}
\end{equation}
where $u$ is the unknown function and the sequence $f= \{f(n)\}_{n \in \Z}$ is the initial datum at time $t=0$.
Other second order differential operators and the associated discrete heat kernels arise when dealing with equations connected with physics, namely the Toda lattice, see~\cite{GrIl,Haine,Iliev}.

We first shall see that the heat semigroup is a positive, Markovian, diffusion semigroup, see Section~\ref{sec:disheatsemigroup}. Then, a maximum principle for the fractional Laplacian is proved, see Theorem~\ref{thm:fractional-Laplacian} below. We can apply the general theory developed by E.~M. Stein \cite{St-rojo} to obtain the boundedness on $\ell^p:=\ell^p(\Z)$, $1 < p< \infty$, of the maximal heat and Poisson operators and the square function. However this general theory does not cover the boundedness on the space~$\ell^1$. By analyzing the kernel of these operators we shall get the results in $\ell^p(w):=\ell^p(\Z, w)$, $1\le p < \infty,$ where $w$ will be an appropriate weight, see Theorem~\ref{thm:semigroups-gg}.

Consider the ``first'' order difference operators
\begin{equation}
\label{factorizacion}
Df(n)=f(n+1)-f(n)\quad\text{ and }\quad
\widetilde{D} f(n)=f(n)-f(n-1),
\end{equation}
that allow factorization of the discrete Laplacian as $\Delta_{\dis} = \widetilde{D} D$.
Having a factorization $L= \widetilde{X}X$ for a general positive laplacian $L$, the ``Riesz transforms'' are usually defined as $X(L)^{-1/2}$ and $\widetilde{X}(L)^{-1/2}$, see \cite{Th1, Th2}. In our case, the operator $(-\Delta_{\dis})^{-1/2}$ is not well defined, so we overcome this difficulty by defining the following ``Riesz transforms'':
\begin{equation}
\label{nuevasriesz}
\mathcal{R}= \lim_{\alpha \to (1/2)^-}D(-\Delta_{\dis})^{-\alpha}
\quad\text{and}\quad
\widetilde{\mathcal{R}}= \lim_{\alpha \to (1/2)^-}\widetilde{D}(-\Delta_{\dis})^{-\alpha}.
\end{equation}
It turns out that these operators are convolution operators with the kernels $\{\frac1{\pi(n+1/2)}\}_{n\in \Z} $ and $\{\frac1{\pi(n-1/2)}\}_{n\in \Z}$. A close analysis of the semigroup and its kernel allows us to define the harmonic conjugate functions and to see that some Cauchy--Riemann equations are satisfied, see Theorem~\ref{thm:Cauchy} below. Moreover, it will be shown that the Riesz transforms are the limits of the conjugate harmonic functions. We provide definitions of these operators and we shall prove that they are effectively bounded on $\ell^2$ by means of the Fourier transform and some abstract results on discrete distributions. We refer the reader to the books by R.~E. Edwards \cite{Edwards} and by E.~M. Stein and G.~Weiss~\cite{StWe} for details.

These are the operators that we consider in the paper:
\begin{enumerate}
\item[(i)] The fractional Laplacian $(-\Delta_{\dis})^\sigma$, $0<\sigma<1$.
\item[(ii)] The maximal heat semigroup $W^*f= \sup_{t\ge0} |W_tf|$
with $W_t = e^{t\Delta_{\dis}}$,
and the maximal Poisson semigroup
$P^*f= \sup_{t\ge0} |P_tf|$
with $P_t = e^{-t\sqrt{-\Delta_{\dis}}}$.
\item[(iii)] The square function
\[
g(f) = \Big( \int_0^\infty |t\partial_t e^{t\Delta_{\dis}}f|^2 \,\frac{dt}{t} \Big)^{1/2}.
\]
\item[(iv)] The Riesz transforms, as defined in~\eqref{nuevasriesz}.
\item[(v)] The conjugate harmonic functions
\[
Q_t f= \mathcal{R}  P_t f  \quad\text{and}\quad  \widetilde{Q}_tf = \widetilde{\mathcal{R}}P_t f, \quad t\ge0.
\]
\end{enumerate}

The main results of this paper are collected in Theorems~\ref{thm:fractional-Laplacian}, \ref{thm:semigroups-gg} and~\ref{thm:Cauchy} below. The first one contains maximum and comparison principles for the fractional Laplacian.

\begin{thm}
\label{thm:fractional-Laplacian}
Let $0<\sigma<1$.
\begin{itemize}
\item[(i)] Let $f\in \ell^2$. Assume that $f\ge 0$ and that $f(n_0)=0$ for some $n_0$. Then $(-\Delta_{\dis})^\sigma f(n_0) \le 0.$
Moreover, $(-\Delta_{\dis})^\sigma f(n_0) =0$ only if $f(n) = 0$ for all $n\in \Z.$
\item[(ii)] Let $f,g \in \ell^2$ be such that $f\ge g$ and $f(n_0)=g(n_0)$ at some $n_0\in \Z.$ Then
$(-\Delta_{\dis})^\sigma f(n_0) \le (-\Delta_{\dis})^\sigma g(n_0)$. Moreover $(-\Delta_{\dis})^\sigma f(n_0) = (-\Delta_{\dis})^\sigma g(n_0) $ only if $f(n) = g(n)$ for all $n\in \Z.$
\end{itemize}
\end{thm}

In order to get mapping properties in $\ell^1$ or $\ell^p(w)$ spaces, we shall use the vector-valued theory of Calder\'on--Zygmund operators in spaces of homogeneous type $(\Z,\mu,|\cdot|)$. Here, $\mu(A)$ is the counting measure, for a set $A\subseteq \Z$, and $|\cdot|$ is the distance in~$\Z$. A weight on $\Z$ is a sequence $w=\{w(n)\}_{n\in\Z}$ of nonnegative numbers.
We recall  that $w$ is a \textit{discrete Muckenhoupt weight for~$\ell^p$}, and we write $w\in A_p$, if there is a constant $C<\infty$ such that, for every pair of integers $M$, $N$, with $M\le N$,
\begin{equation*}
\Big(\sum_{k=M}^{N}w(k)\Big) \Big(\sum_{k=M}^N w(k)^{-1/(p-1)}\Big)^{1/(p-1)}\le C (N-M+1)^p, \quad 1<p<\infty,
\end{equation*}
and
\[
\Big(\sum_{k=M}^{N}w(k)\Big)\sup_{k\in [M,N]}w(k)^{-1}\le C(N-M+1), \quad p=1,
\]
see \cite[Section~8]{HuMuWh}.

The second result concerns mapping $\ell^p(w)$ properties for the maximal heat and Poisson semigroups and square functions.

\begin{thm}
\label{thm:semigroups-gg}
Let $w\in  A_p$, $1\le p<\infty$. Then the operators $W^*$, $P^*$, and $g$ are norms of vector-valued Calder\'on--Zygmund operators in the sense of the space of homogeneous type $(\Z,\mu,|\cdot|)$.
Therefore, all these operators are bounded from $\ell^p(w)$ into itself for $1<p<\infty$
and also from $\ell^1(w)$ into weak-$\ell^1(w)$.
\end{thm}

Finally, Riesz transforms can be seen as limits of conjugate harmonic functions; the latter are bounded on $\ell^p(w)$ and satisfy some Cauchy--Riemann equations.

\begin{thm}
\label{thm:Cauchy}
Let $f\in\ell^p(w)$, $1\le p < \infty$, with $w\in A_p$. Then
\begin{itemize}
\item[(i)] The operators $ Q^*f = \sup_{t\ge0} Q_t f$ and $ \widetilde{Q}^* f= \sup_{t\ge0}\widetilde{Q}_t f$ are bounded
from $\ell^p(w), 1<p< \infty$, into itself and from $\ell^1(w)$ into weak-$\ell^1(w).$
\item[(ii)] The operators $Q_t$, $\widetilde{Q}_t$ and $P_t$ satisfy the Cauchy--Riemann type equations
\[
\begin{cases}
\partial_t(Q_tf) = -D(P_tf), \\
 \widetilde{D}(Q_tf) = \partial_t(P_tf);
\end{cases}
\quad \quad\begin{cases}
\partial_t(\widetilde{Q}_tf)
= -\widetilde{D}(P_tf),\\
 D(\widetilde{Q}_tf) = \partial_t(P_tf).
\end{cases}
\]
Moreover, $\partial^2_{tt} Q_tf(n) + \Delta_{\dis} Q_tf(n) =0$
and $\partial^2_{tt} \widetilde{Q}_tf(n)  + \Delta_{\dis} \widetilde{Q}_tf(n) =0$.
\item[(iii)] We have
\[
\lim_{t\to 0} Q_tf(n) = \mathcal{R} f(n) \quad \text{and}\quad \lim_{t\to 0} \widetilde{Q}_tf(n) = \widetilde{\mathcal{R}} f(n),
\]
for $n \in \mathbb{Z}$.
The limits also holds in $\ell^p(w)$ sense for $1<p<\infty$.
\end{itemize}
\end{thm}

The study on $\ell^p$ of discrete operators of harmonic analysis was initiated by M.~Riesz~\cite{Riesz}. Apart from the $L^p(\R)$ boundedness of the Hilbert transform, he showed the $\ell^p$ boundedness of its discrete analogue. Later, A.~P. Calder\'on and A.~Zygmund \cite{Calderon-Zygmund} noticed that $L^p(\R^d)$ boundedness of singular integrals implied the $\ell^p(\Z^d)$ boundedness of their discrete analogues. R.~Hunt, B.~Muckenhoupt and R.~Wheeden proved weighted inequalities for the discrete Hilbert and discrete maximal operator in the one-dimensional case~\cite{HuMuWh}.
In the last few years, several works have been developed for non convolution discrete analogues of continuous operators; some contributions were made by I. Arkhipov and K.~I. Oskolkov, see~\cite{ArOs},
J. Bourgain \cite{Bourgain4} and Stein and S.~Wainger~\cite{StWa1,StWa2}.
For references concerning research on discrete analogues see~\cite{Pierce-tesis}, where a brief history and a nice exposition of recent progress can be found.

In this work we show one-dimensional results on some operators of the harmonic analysis. Results concerning multidimensional discrete fractional integrals and multidimensional discrete Riesz transforms will appear elsewhere.


The paper is organized as follows. In Section~\ref{sec:disheatsemigroup} the theory of semigroups for the operator~\eqref{primercalor} is developed. In Section~\ref{sec:maximum} we prove Theorem~\ref{thm:fractional-Laplacian} concerning the maximum and comparison principles for the fractional Laplacian, and Theorem~\ref{thm:semigroups-gg} is proved in Sections~\ref{sec:semigroups} and~\ref{sec:gg}. Sections~\ref{sec:Riesz} and~\ref{sec:Cauchy} contain the definition on $\ell^2$ of Riesz transforms and conjugate harmonic functions and the proof of Theorem~\ref{thm:Cauchy}. Finally, note that the modified Bessel functions~$I_k$ are used often throughout the paper. They always involve rather sophisticated and technical computations; then, to make the paper more readable we collect in Section~\ref{sec:preliminaries} all the properties that we need concerning these functions.

\section{The discrete heat and Poisson semigroups}
\label{sec:disheatsemigroup}

Let
\begin{equation}
\label{eq:heatkernel}
  G(k,t) = e^{-2t} I_{k}(2t), \quad  k \in \Z.
\end{equation}
Observe that, by \eqref{eq:simetriak}, $G(-k,t) = G(k,t)$. Consider  the operator
\begin{equation}
\label{eq:heatconv}
  W_tf(n) = \sum_{m\in\Z} G(n-m,t) f(m), \quad t>0.
\end{equation}
We shall prove in Proposition~\ref{prop:heat-is-semigroup} that $\{W_t\}_{t \ge 0}$ is a \textit{positive Markovian diffusion semigroup}, see \cite[Chapter 3, p.~65]{St-rojo}.

Some previous definitions are needed. Let $\T\equiv\R/(2\pi\Z)$ be the one-dimensional torus. We identify the torus with the interval $(-\pi,\pi]$ and functions on $\T$ with $2\pi$-periodic functions on~$\R$. Also, integration over $\T$ can be described in terms of Lebesgue integration over $(-\pi,\pi]$.
Given a sequence $f \in \ell^1$ we can define its Fourier transform $\mathcal{F}_{\Z}(f)(\theta)= \sum_n  f(n) e^{in\theta}$, $\theta \in \T$.
It is well known that the operator $f \mapsto \mathcal{F}_{\Z}(f)$ can be extended as an isometry from $\ell^2$ into $L^2(\T)$, where the inverse operator $ \mathcal{F}_{\Z}^{-1}$ is given by
\begin{equation}\label{inversion}
\mathcal{F}_\Z^{-1}(\varphi)(n) = \frac1{2\pi}\int_{\T} \varphi(\theta) e^{-in\theta} \,d\theta.
\end{equation}

\begin{prop}
\label{prop:heat-is-semigroup}
Let $f\in \ell^\infty.$
The family $\{W_t\}_{t \ge 0}$ satisfies
\begin{enumerate}
\item[(i)] $W_0f = f$.
\item[(ii)] $W_{t_1}W_{t_2}f = W_{t_1+t_2}f$.
\item[(iii)] If $f\in \ell^2$ then $W_tf \in \ell^2$ and $\lim_{t\to 0}W_tf= f$ in~$\ell^2$.
\item[(iv)] (Contraction property) $\|W_t f\|_{\ell^p} \le \|f\|_{\ell^p}$ for $1 \le p \le +\infty$.
\item[(v)] (Positivity preserving) $W_t f \ge 0$ if $f \ge 0$, $f\in \ell^2$.
\item[(vi)] (Markovian property) $W_t 1 = 1$.
\end{enumerate}
\end{prop}

\begin{proof}
By using the identity~\eqref{eq:sumIk}, we have
\[
\biggl|\sum_{m\in\Z} e^{-2t} I_m(2t) f(n-m) \biggr|
\le \|f\|_{\ell^\infty} \sum_{m\in\Z} e^{-2t} I_m(2t) = \|f\|_{\ell^\infty}.
\]
Therefore $W_t$ is well defined in $\ell^\infty.$ Now~\eqref{eq:Ik0} gives~(i), since
\[
  W_0f(n) = \sum_{m\in\Z} G(n-m,0) f(m) = \sum_{m\in\Z} e^{0} I_{n-m}(0) f(m) = f(n).
\]

Concerning~(ii), we use~\eqref{eq:Neumann}, so that
\begin{align*}
  \sum_{k \in\Z} &G(n-k,t_1) G(k-m,t_2) \\
  &= \sum_{k \in\Z} e^{-2t_1} I_{n-k}(2t_1) e^{-2t_2} I_{k-m}(2t_2) \\
  &= e^{-2(t_1+t_2)} I_{n-m}(2(t_1+t_2))
  = G(n-m,t_1+t_2).
\end{align*}
Then,
\begin{align*}
  W_{t_1}W_{t_2}f(n) &= W_{t_1} \left(\sum_{m\in\Z} G( \cdot - m,t_2) f(m) \right) (n) \\
  &= \sum_{k\in\Z} G(n-k,t_1) \left(\sum_{m\in\Z} G(k-m,t_2) f(m) \right) \\
  &= \sum_{m\in\Z} \left(\sum_{k\in\Z} G(n-k,t_1) G(k-m,t_2) \right) f(m) \\
  &= \sum_{m\in\Z} G(n-m,t_1+t_2)f(m) = W_{t_1+t_2}f(n)
\end{align*}
and the proof is finished. We skip the proof of~(iii) for a while.
For~(iv), Minkowski's integral inequality yields
\begin{align*}
  \|W_t f\|_{\ell^p}
  &= \left( \sum_{n\in\Z} |W_tf(n)|^p \right)^{1/p}
  = \left( \sum_{n\in\Z} \biggl| \sum_{m\in\Z} e^{-2t} I_m(2t) f(n-m) \biggr|^p \right)^{1/p} \\
  &\le \sum_{m\in\Z} \left( \sum_{n\in\Z} |e^{-2t} I_m(2t) f(n-m)|^p \right)^{1/p} \\
  &= \sum_{m\in\Z} e^{-2t} I_m(2t) \left( \sum_{n\in\Z} |f(n-m)|^p \right)^{1/p} \\
  &= \sum_{m\in\Z} e^{-2t} I_m(2t) \|f\|_{\ell^p} = \|f\|_{\ell^p},
\end{align*}
where we have used both~\eqref{eq:Ik>0} and~\eqref{eq:sumIk}. Part~(v) follows from~\eqref{eq:Ik>0}. Part~(vi) is obtained by using~\eqref{eq:sumIk}:
\[
  W_t 1(n) = \sum_{m\in\Z} e^{-2t} I_{m-n}(2t) \cdot 1 = \sum_{m\in\Z} e^{-2t} I_{m}(2t) = 1
\]
for every $n \in \Z$.

Finally, we prove (iii). We have already proved the boundedness in $\ell^2$ and we only need to care about the limit. Observe that we can write $W_tf(n) = (G(\cdot,t)\ast f) (n)$, where the convolution is performed on~$\Z$ (i.e., $g\ast f(n) = \sum_{m} g(n-m)f(m)$). Moreover,
\[
\mathcal{F}_{\Z}(G(\cdot,t)\ast f)(\theta)
= \mathcal{F}_{\Z}(G(\cdot,t))(\theta)\mathcal{F}_{\Z}(f)(\theta).
\]
We compute $\mathcal{F}_{\Z}(G(\cdot,t))(\theta)$. By~\eqref{eq:heatkernel}
and the formula (see \cite[p.~456]{Prudnikov})
\[
\int_0^{\pi}e^{z\cos\theta}\cos m\theta\,d\theta = \pi I_m(z),\quad |\arg z|<\pi,
\]
we have
\begin{align*}
G(m,t) = e^{-2t}I_m(2t)
&= \frac{e^{-2t}}{\pi} \int_0^{\pi}e^{2t\cos\theta}\cos m\theta\,d\theta\\
&= \frac{e^{-2t}}{2\pi} \int_{-\pi}^{\pi}e^{2t\cos\theta}(\cos m\theta-i\sin m\theta)\,d\theta\\
&= \frac{e^{-2t}}{2\pi} \int_{-\pi}^{\pi}e^{2t\cos\theta}e^{-im\theta}\,d\theta.
\end{align*}
In view of the inversion formula~\eqref{inversion}
we conclude that
\begin{equation*}
\mathcal{F}_{\Z}(G(\cdot,t))(\theta) = e^{-2t(1-\cos\theta)}
= e^{-4t\sin^2\frac{\theta}{2}}.
\end{equation*}
Therefore,
\[
\lim_{t\to 0} \|W_tf-f\|_{\ell^2}
= \lim_{t\to 0} \|(e^{-4t\sin^2\frac{(\cdot)}{2}} -1) \mathcal{F}_{\Z}(f)(\theta)\|_{L^2(\T)} = 0.
\qedhere
\]
\end{proof}

\begin{rem}
Observe that, for $0<t < 1/8$, we have
\begin{equation*}
\|W_tf-f\|_{\ell^2}
= \|(e^{-4t\sin^2\frac{(\cdot)}{2}} -1)\mathcal{F}_{\Z}(f)(\theta)\|_{L^2(\T)}
\le C t \|f\|_{\ell^2}.
\end{equation*}
In particular, for $f\in \ell^2$ we have
\begin{equation}\label{util}
|W_tf(n)-f(n)| \le Ct\|f\|_{\ell^2}
\end{equation}
for every $n \in \Z$.
\end{rem}

\begin{prop}
\label{prop:solutionHeat}
Let $f\in \ell^\infty$. Then,
\[
  u(n,t) = \sum_{m\in\Z} e^{-2t} I_{n-m}(2t) f(m),
  \quad t>0,\quad n\in \Z,
\]
is a solution of the equation~\eqref{eq:calorZ}.
\end{prop}

\begin{proof}
We just use~\eqref{eq:derivadaCalort} and~\eqref{eq:Ik0}.
\end{proof}

\begin{rem}\label{Poisson}
It can be checked that the Poisson operator defined, via the subordination formula
\begin{equation} \label{subordination}
e^{-\beta t}=\frac{1}{\sqrt\pi} \int_0^{\infty} \frac{e^{-u}}{\sqrt u}e^{-\frac{t^2 \beta^2}{4u}}\,du
= \frac{t}{2\sqrt\pi} \int_0^{\infty} \frac{e^{- t^2/(4v)}}{\sqrt {v}}e^{-v\beta^2}\,\frac{dv}{v},
\end{equation}
by
\begin{equation}
\label{eq:Poissonsub}
P_tf(n) = \frac{1}{\sqrt\pi} \int_0^{\infty} \frac{e^{-u}}{\sqrt u}W_{t^2/(4u)}f(n)\,du
= \frac{t}{2\sqrt\pi} \int_0^{\infty} \frac{e^{- t^2/(4v)}}{\sqrt {v}}W_vf(n)\,\frac{dv}{v}
\end{equation}
satisfies the ``Laplace'' equation
\[
\partial^2_{tt} P_tf(n) + \Delta_{\dis} P_tf(n)=0.
\]
\end{rem}

\section{Maximum principle for fractional powers of the discrete Laplacian}
\label{sec:maximum}

In this section we prove Theorem~\ref{thm:fractional-Laplacian}.
Given $0<\sigma < 1$, we define
\begin{equation}
\label{fraccionario}
(-\Delta_{\dis})^\sigma f (n)= \frac1{\Gamma(-\sigma)} \int_0^\infty
(e^{t\Delta_{\dis}}f(n) -f(n) )\frac{dt}{t^{1+\sigma}}, \quad f \in\ell^2.
\end{equation}
Notice that, by~\eqref{eq:heatconv}, \eqref{eq:heatkernel} and~\eqref{eq:sumIk}, the integrand of this operator can be written as
\begin{equation}
\label{fracaux}
  e^{t\Delta_{\dis}}f(n) - f(n) = \sum_{m\in \Z} G(n-m,t) \big(f(m)-f(n) \big).
\end{equation}

\begin{proof}[Proof of Theorem~\ref{thm:fractional-Laplacian}]
Observe that $(-\Delta_{\dis})^\sigma f $ is well defined in formula~\eqref{fraccionario}; indeed,
if we decompose
\begin{multline*}
\int_0^\infty |e^{t\Delta_{\dis}}f(n) -f(n)| \frac{dt}{t^{1+\sigma}} \\
=
\int_0^{1/8} |e^{t\Delta_{\dis}}f(n) -f(n)| \frac{dt}{t^{1+\sigma}}
+
\int_{1/8}^\infty |e^{t\Delta_{\dis}}f(n) -f(n)| \frac{dt}{t^{1+\sigma}},
\end{multline*}
both integrals are finite by~\eqref{util} and~\eqref{fracaux}.
Then,
\begin{align*}
(-\Delta_{\dis})^\sigma f (n_0)
&= \frac1{\Gamma(-\sigma)}
\int_0^\infty \sum_{m\in \Z} G(n_0-m,t)\big(f(m)-f(n_0) \big) \frac{dt}{t^{1+\sigma}} \\
&= \frac1{\Gamma(-\sigma)}
\int_0^\infty \sum_{m\in \Z} G(n_0-m,t) f(m) \frac{dt}{t^{1+\sigma}}.
\end{align*}
The positivity of $G(n,t)$ gives the maximum principle stated in (i).

The comparison principle for $(-\Delta)^\sigma$ in (ii) is an immediate consequence of the maximum principle.
\end{proof}

\section{Heat and Poisson semigroups as $\ell^\infty$ norms of Calder\'on--Zygmund operators}
\label{sec:semigroups}

In this section, we give the proof of Theorem~\ref{thm:semigroups-gg} for the operators $W^*$ and~$P^*$. The key point is to obtain estimates for the kernels, which are contained in Proposition~\ref{prop:heat-kernel-estimates} below.

In the proof of such proposition, apart from facts concerning modified Bessel functions (see Section~\ref{sec:preliminaries}), we will frequently use the well-known fact
\begin{equation}
\label{eq:gambas}
\frac{\Gamma(z+r)}{\Gamma(z+t)}\sim z^{r-t}, \quad z>0, \quad r,t\in \R.
\end{equation}
Also that, for $\eta>0$ and $\gamma\ge0$,
\begin{equation}
\label{functionh}
(1-r)^\eta r^{\gamma}\le \left(\frac{\gamma}{\gamma+\eta}\right)^\gamma,
\quad\text{when}~0<r<1.
\end{equation}
Actually, both estimates above will be needed throughout the rest of the paper.

\begin{prop}
\label{prop:heat-kernel-estimates}
Let $T(m,t)$ be either the  discrete heat kernel $G(m,t)$ or the Poisson kernel.
The following estimates are satisfied:
\[
\sup_{t\ge0}|T(m,t)|\le \frac{C_1}{|m|+1} \qquad \hbox{and} \qquad
\sup_{t\ge0}|T(m+1,t)-T(m,t) |\le \frac{C_2}{m^2+1},
\]
where $C_1$ and $C_2$ are constants independent of $m\in \Z$.
\end{prop}

\begin{proof}
We start with the heat kernel. Observe that, by \eqref{eq:Ik}, \eqref{eq:asymptotics-zero} and \eqref{eq:asymptotics-infinite}
\[
\sup_{t\ge0}G(0,t)=\sup_{t\ge0}e^{-2t}I_0(2t)\le C.
\]
Moreover, by \eqref{eq:simetriak}, we can assume that $m>0$.
By~\eqref{eq:heatkernel} and~\eqref{eq:Schlafli}, we can write
\begin{align*}
G(m,t) &= e^{-2t}I_m(2t)=e^{-2t}\frac{(2t)^m}{\sqrt{\pi}\,2^m\Gamma(m+1/2)}
\int_{-1}^1e^{-2ts}(1-s^2)^{m-1/2}\,ds\\
&= \frac{t^m}{\sqrt{\pi}\,\Gamma(m+1/2)} \int_{-1}^1e^{-2t(1+s)}(1-s^2)^{m-1/2}\,ds\\
&= \frac{2t^m}{\sqrt{\pi}\,\Gamma(m+1/2)} \int_{0}^te^{-4w}
\big(\tfrac{2w}{t}\big)^{m-1/2} \big(2\big(1-\tfrac{w}{t}\big)\big)^{m-1/2}\,
\tfrac{dw}{t}\\
&= \frac{t^{-1/2}4^m}{\sqrt{\pi}\,\Gamma(m+1/2)} \int_{0}^te^{-4w}w^{m-1}w^{1/2}
\big(1-\tfrac{w}{t}\big)^{m-1/2}\,dw\\
&\le \frac{4^m}{\sqrt{\pi}\,\Gamma(m+1/2)m^{1/2}}\int_{0}^te^{-4w}w^{m-1}\,dw
\le C\frac{\Gamma(m)}{\Gamma(m+1/2)m^{1/2}}\sim \frac{1}{m},
\end{align*}
where we have used~\eqref{functionh} with $r=\frac{w}{t}$, $\eta=m-1/2$ and $\gamma=1/2$, and~\eqref{eq:gambas}.

Concerning smoothness estimates, by using~\eqref{eq:difIs}, and proceeding as in the growth estimate, we arrive at
\begin{align*}
|DG(m,t)|
&=\frac{2t^{-3/2}4^m}{\sqrt{\pi}\,\Gamma(m+1/2)} \int_{0}^te^{-4w}w^{m-1}w^{3/2}
\big(1-\tfrac{w}{t}\big)^{m-1/2}\,dw\\
&\le 2\frac{4^m}{\sqrt{\pi}\,\Gamma(m+1/2)(m+1)^{3/2}} \int_{0}^te^{-4w}w^{m-1}\,dw\\
&\le C\frac{\Gamma(m)}{\Gamma(m+1/2)(m+1)^{3/2}} \sim \frac{1}{m^2},
\end{align*}
where we have used~\eqref{functionh} with $r=\frac{w}{t}$, $\eta=m-1/2$ and $\gamma=3/2$, and~\eqref{eq:gambas}.

By using the subordination formula~\eqref{eq:Poissonsub} we get that the Poisson kernel satisfies the desired estimates.
\end{proof}

\begin{proof}[Proof of Theorem~\ref{thm:semigroups-gg} for $W^*$ and $P^*$]
Let $\sup_{t\ge0} |T_t|$ be either $W^*$ (with $T_t=W_t$) or $P^*$ (with $T_t=P_t$). From Proposition~\ref{prop:heat-is-semigroup}, Stein's Maximal Theorem of diffusion semigroups (see \cite[Chapter~III, Section~2]{St-rojo}) establishes that $\sup_{t\ge0} T_t$ is bounded from $\ell^p$, $1< p< \infty,$ into itself. Then the vector valued operator  $f \mapsto \widetilde{T}f= \{T_t f\}_t$ is bounded from $\ell^p$ into $\ell^p_{L^\infty}$ for $1<p<\infty$. Here, $\ell^p_{\mathbb{B}}$ is the space of functions such that $\big(\sum_n\|f(n)\|^p_{\mathbb{B}}\big)^{1/p}$, with $\mathbb{B}$ a Banach space.

On the other hand, by Proposition~\ref{prop:heat-kernel-estimates} the kernel of the operator $\widetilde{T}$ satisfies the so-called Calder\'on--Zygmund estimates. Therefore, by the general theory of vector-valued Calder\'on--Zygmund operators in spaces of homogeneous type (see \cite{Rubio-Ruiz-Torrea,Ruiz-Torrea}), the operator $\widetilde{T}$
is bounded from $\ell^p(w)$ into $\ell^p_{L^\infty}(w)$, for $1<p<\infty$ and $w\in A_p$, and
from $\ell^1(w) $ into weak-$\ell^1_{L^\infty}(w)$, for $w\in A_1$.
\end{proof}

As a standard corollary of Theorem~\ref{thm:semigroups-gg} we have the following:

\begin{coro}
Let $\sup_{t\ge0} |T_t|$ be either $W^*$ (with $T_t=W_t$) or $P^*$ (with $T_t=P_t$). Then
$\lim_{t\to 0} T_t f(n) = f(n)$ for every $n$ and every function $f\in \ell^p(w)$, for $w\in A_p$, $1\le p < \infty$.
\end{coro}

\section{The discrete $g$-function}
\label{sec:gg}

In this section, we prove Theorem~\ref{thm:semigroups-gg} for the discrete $g$-function
\[
gf(n) = \Big(\int_0^{\infty} |t\partial_t W_tf(n)|^2 \,\tfrac{dt}{t}\Big)^{1/2}.
\]
Let $\mathbb{B}=L^2((0,\infty),\frac{dt}{t})$ be a Banach space.
We shall first prove the following appropriate vector-valued kernel estimates.

\begin{prop}
\label{prop:gg-kernel-estimates}
Let $G(m,t)$ be the discrete heat kernel.
The following estimates are satisfied:
\begin{equation*}
\|t\,\partial_t G(m,t)\|_{\mathbb{B}}\le \frac{C_1}{|m|+1} \quad\text{and}\quad
\|D(t\,\partial_t G(m,t))\|_{\mathbb{B}}\le \frac{C_2}{m^2+1},
\end{equation*}
where $C_1$ and $C_2$ are constants independent of $m\in \Z$.
\end{prop}

\begin{proof}
By an analogous reasoning to that in Proposition~\ref{prop:heat-kernel-estimates}, we can assume that $m>0$.
We begin with the growth estimate. Observe that, by~\eqref{eq:derivadaCalort} and~\eqref{eq:difIs3terms}, one has
\begin{align*}
\|t\partial_tG(m,t)\|_{\mathbb{B}}
&= \|te^{-2t} (I_{m+1}(2t) - 2I_{m}(2t) + I_{m-1}(2t))\|_{\mathbb{B}}\\
&\le C\frac{S_1^{1/2}+S_2^{1/2}}{\Gamma(m-1/2)},
\end{align*}
where
\[
S_1 := \int_0^{\infty}t\Big(t^{m-2}e^{-2t}
\int_{-1}^1e^{-2ts}s(1-s^2)^{m-3/2}\,ds\Big)^2\,dt,
\]
and
\[
S_2 := \int_0^{\infty}t\Big(t^{m-1}e^{-2t}
\int_{-1}^1e^{-2ts}(1+s)^2(1-s^2)^{m-3/2}\,ds\Big)^2\,dt.
\]
Concerning $S_1$,
\begin{align*}
S_1 &\le \int_0^{\infty}t^{2m-3}e^{-4t}\int_{-1}^1e^{-2ts}(1-s^2)^{m-3/2}\,ds
\int_{-1}^1e^{-2tu}(1-u^2)^{m-3/2}\,du\,dt \\
&= \int_{-1}^1\int_{-1}^1(1-s^2)^{m-3/2}(1-u^2)^{m-3/2}
\int_0^{\infty}t^{2m-3}e^{-2t(2+s+u)}\,dt\,ds\,du \\
&= \frac{\Gamma(2m-2)}{2^{2m-2}}\int_{-1}^1\int_{-1}^1
\frac{(1-s^2)^{m-3/2}(1-u^2)^{m-3/2}}{(2+s+u)^{2m-2}}\,du\,ds \\
&= \Gamma(2m-2)\int_{0}^1
\int_{0}^1 \frac{x^{m-3/2}(1-x)^{m-3/2}y^{m-3/2}(1-y)^{m-3/2}}
{(x+y)^{2m-2}}\,dx\,dy
\\
&\le \Gamma(2m-2)\int_{0}^1y^{m-3/2}(1-y)^{m-3/2}
\int_{0}^1 \frac{x^{m-3/2}}
{(x+y)^{2m-2}}\,dx\,dy
\\
&= \Gamma(2m-2)\int_{0}^1(1-y)^{m-3/2}\int_{0}^{1/y}
\frac{w^{m-3/2}}
{(1+w)^{2m-2}}\,dw\,dy
\\
&\le \frac{\Gamma(2m-2)\Gamma(m-1/2)\Gamma(m-3/2)}{\Gamma(2m-2)}
\int_{0}^1(1-y)^{m-3/2}\,dy\int_0^{\infty}\frac{w^{m-3/2}}
{(1+w)^{2m-2}}\,dw \\*
&= \frac{\Gamma(m-1/2)^2}{(m-1/2)^2},
\end{align*}
so that
\[
\frac{S_1^{1/2}}{\Gamma(m-1/2)}\le\frac{C}{(m-1/2)}.
\]
For $S_2$, we proceed analogously, and we get
\begin{align*}
S_2 &\le \Gamma(2m)\frac{\Gamma(m+3/2)\Gamma(m-3/2)}
{\Gamma(2m)}\int_{0}^1y^{m+1/2}(1-y)^{m-3/2} \,dy \\
&= \frac{\Gamma(m+3/2)\Gamma(m-3/2)\Gamma(3)\Gamma(m-1/2)}{\Gamma(m+5/2)},
\end{align*}
hence
\[
\frac{S_2^{1/2}}{\Gamma(m-1/2)}\le\frac{C}{(m-1/2)}.
\]
By pasting together the estimates for $S_1$ and $S_2$, the bound for $\|t\partial_tG(m,t)\|_{\mathbb{B}}$ follows.

Now we pass to the smoothness estimates. By~\eqref{eq:heatkernel} and~\eqref{eq:difIs4terms}, we have
\begin{align*}
\|D(t\partial_tG(m,t))\|_{\mathbb{B}}
&= \|te^{-2t} (I_{m+2}(2t) - 3I_{m+1}(2t) + 3I_{m}(2t)-I_{m-1}(2t))\|_{\mathbb{B}}\\
&\le C \,\frac{T_1^{1/2}+T_2^{1/2}+T_3^{1/2}}{\Gamma(m-1/2)},
\end{align*}
where
\[
T_1 = \int_0^{\infty}t\Big(t^{m-3}e^{-2t}
\int_{-1}^1e^{-2ts}s(1-s^2)^{m-3/2}\,ds\Big)^2\,dt,
\]
\[
T_2 = \int_0^{\infty}t\Big(t^{m-2}e^{-2t}
\int_{-1}^1e^{-2ts}s(1+s)(1-s^2)^{m-3/2}\,ds\Big)^2\,dt,
\]
and
\[
T_3 = \int_0^{\infty}t\Big(t^{m-1}e^{-2t}
\int_{-1}^1e^{-2ts}(1+s)^3(1-s^2)^{m-3/2}\,ds\Big)^2\,dt.
\]
In order to treat each term we follow the same procedure as in the growth estimates. Hence, concerning $T_1$, we get
\begin{align*}
T_1&
\le \Gamma(2m-4)\frac{\Gamma(m-1/2)\Gamma(m-7/2)}{\Gamma(2m-4)}\int_{0}^1y^{2}
(1-y)^{m-3/2}
\,dy\\
&=\frac{\Gamma(m-7/2)\Gamma(3)\Gamma(m-1/2)^2}{\Gamma(m+5/2)}.
\end{align*}
For $T_2$ we obtain
\begin{align*}
T_2
&\le \Gamma(2m-2)\frac{\Gamma(m+1/2)\Gamma(m-5/2)}{\Gamma(2m-2)}\int_{0}^1y^{2}
(1-y)^{m-3/2}
\,dy\\
&=\frac{\Gamma(m+1/2)\Gamma(m-5/2)\Gamma(3)\Gamma(m-1/2)}{\Gamma(m+5/2)}.
\end{align*}
Finally,
\begin{align*}
T_3
&\le \Gamma(2m)\frac{\Gamma(m+5/2)\Gamma(m-5/2)}{\Gamma(2m)}\int_{0}^1y^{4}(1-y)^{m-3/2}\,dy
\\
&=\frac{\Gamma(m+5/2)\Gamma(m-5/2)}{\Gamma(m-1/2)^2}
\frac{\Gamma(5)\Gamma(m-1/2)}{\Gamma(m+9/2)}.
\end{align*}
By pasting together the estimates for $T_1$, $T_2$ and $T_3$, and~\eqref{eq:gambas}, we get the desired bound.
\end{proof}

\begin{proof}[Proof of Theorem~\ref{thm:semigroups-gg} for $g$]
Since $W_t$ is a diffusion semigroup, $g$ is bounded from $\ell^2$ into itself (see \cite[p.~74]{St-rojo}). In order to extend this result to weighted inequalities and to the range $1\le p<\infty$ we observe that $g$ can be seen as a vector-valued operator taking values in the Banach space $\mathbb{B}=L^2((0,\infty),\frac{dt}{t})$, so that $gf=\|t\partial_tW_tf\|_{\mathbb{B}}$. Indeed,
\[
gf(n)=\|S_tf(n)\|_{\mathbb{B}},
\]
where
\[
S_tf(n) = \sum_{m\in \Z}t \frac{\partial}{\partial t}G(n-m,t)f(n).
\]
Now, by Proposition~\ref{prop:gg-kernel-estimates} and the general theory of vector-valued Calder\'on--Zygmund operators, Theorem~\ref{thm:semigroups-gg} holds for the square function~$g$.
\end{proof}

\section{Riesz transforms: $\ell^2$ definition and mapping properties in $\ell^p(w)$ spaces}
\label{sec:Riesz}

In this section we define and study the discrete Riesz transforms, and we see that they coincide essentially with the discrete Hilbert transform.

A first attempt to define properly the discrete Riesz transform $\mathcal{R}$ could be by means of the heat semigroup, as done with the fractional Laplacian and square functions. In this way, we introduce the discrete fractional integral $(-\Delta_{d})^{-\alpha}$ for $0< \alpha < 1/2$. We use the formula
\begin{equation}
\label{eq:fractional-integral}
(-\Delta_{\dis})^{-\alpha} = \frac1{\Gamma(\alpha)} \int_0^\infty  e^{t\Delta_{\dis}} t^{\alpha} \,\frac{dt}{t}.
\end{equation}
Observe that, by \eqref{eq:heatkernel} and asymptotics \eqref{eq:asymptotics-infinite} for $t\to \infty$, the integral above is not absolutely convergent for $\alpha=1/2$.
Then, although formally we can write
\[
\mathcal{R}=D(-\Delta_{d})^{-1/2},
\]
we cannot define the Riesz transform in $\ell^2$ by using \eqref{eq:fractional-integral} with $\alpha=1/2$.

In order to define properly these operators in $\ell^2$, we need to develop an abstract theory of discrete distributions and Fourier transform (see \cite{Edwards, StWe} for details). Recall the definition of Fourier transform $\mathcal{F}_{\Z}$ given in Section~\ref{sec:disheatsemigroup}. We shall use the notation
\[
\mathcal{F}_\T(\varphi)(n):=\mathcal{F}^{-1}_\Z(\varphi)(n).
\]
Analogously, we can keep the notation $\mathcal{F}_\Z(f)(\theta):=\mathcal{F}^{-1}_\T(f)(\theta)$.
Let us consider the class of sequences  $\{c_n\}_{n \in \Z} $ such that for each $k\in \N$ there exists $C_k$ with $\sup_{n\in \Z}  |n|^k |c_n| < C_k.$ We denote this space of sequences by  $\mathcal{S}(\Z)$. With the obvious family of seminorms, the Fourier transform produces an isomorphism between the space  $\mathcal{S}(\Z)$ and  the  space of $C^\infty(\T)$ functions (every function in $C^\infty(\T)$ is associated with the sequence of its Fourier coefficients). This isomorphism  allows to define the Fourier transform in the spaces of distributions  $\big( \mathcal{S}(\Z) \big)'$ and $\big( C^\infty(\T) \big)'$  as follows:
\begin{equation}\label{FS}
\big\langle \mathcal{F}_\Z(\Lambda), \varphi \big\rangle_{\T} = \big\langle\Lambda,  \mathcal{F}_\T (\varphi) \big\rangle_{\Z},
\quad \Lambda \in \big( \mathcal{S}(\Z) \big)'  \text{ and }  \varphi  \in C^\infty(\T),
\end{equation}
and
\begin{equation}\label{FS2}
\big\langle \mathcal{F}_\T(\Phi), f \big\rangle_{\Z} = \big\langle\Phi,  \mathcal{F}_\Z (f) \big\rangle_{\T},
\quad \Phi \in \big(C^\infty(\T) \big)'  \text{ and }  f  \in \mathcal{S}(\Z),
\end{equation}
see \cite[Chapter~12]{Edwards}.  We shall need some extra definitions of different actions over sets of distributions.

The convolution of  a distribution $\Lambda \in \big( \mathcal{S}(\Z) \big)'$ with a function $f \in \mathcal{S}(\Z) $ is given by
\begin{equation}\label{convolution}
\big\langle \Lambda * f, g \big\rangle_{\Z} = \big\langle \Lambda, \tilde{f}*g\big\rangle_{\Z} ,
\quad f,g \in \mathcal{S}(\Z) \text{ and } \tilde{f} (n) = f(-n).
\end{equation}
The multiplication of a distribution $\mathcal{U}\in \big( C^\infty(\T) \big)'$ by a function $\varphi \in C^\infty(\T)$ is given by
\begin{equation}\label{producto}
\big\langle \mathcal{U} \psi, \varphi \big\rangle_{\T} = \big\langle  \mathcal{U} , \psi \varphi \big\rangle_{\T} ,
\quad \psi, \varphi  \in C^\infty(\T).
\end{equation}
Observe that by using~\eqref{FS}, \eqref{convolution}, \eqref{producto} and some other standard properties of Fourier transforms, we get
\begin{equation*}
\begin{aligned}
\big\langle \mathcal{F}_\Z(\Lambda* f), \varphi \big\rangle_{\T}
&= \big\langle \Lambda* f, \mathcal{F}_\T(\varphi) \big\rangle_{\Z}
= \big\langle \Lambda, \tilde{f}* \mathcal{F}_\T(\varphi) \big\rangle_{\Z} \\
&= \big\langle \Lambda, \mathcal{F}_\T[  \mathcal{F}_\T^{-1}(\tilde{f})  \varphi] \big\rangle_{\Z}
= \big\langle  \mathcal{F}_\Z (\Lambda), \mathcal{F}_\T^{-1}(\tilde{f})  \varphi \big\rangle_{\T}  \\
&= \big\langle  \mathcal{F}_\Z (\Lambda)  \mathcal{F}_\T^{-1}(\tilde{f}),  \varphi \big\rangle_{\T}
=  \Big\langle  \mathcal{F}_\Z (\Lambda)  \mathcal{F}_\Z(f), \varphi \big\rangle_{\T}.
\end{aligned}
\end{equation*}
Another fact that we need is the following. A linear operator $L: \big( \mathcal{S}(\Z) \big)' \to \big( \mathcal{S}(\Z) \big)'$ is bounded if and only if the operator $\mathcal{F}_\Z \circ L \circ \mathcal{F}_\Z^{-1}$ is bounded from
$\big( C^\infty(\T) \big)'$ into  $\big( C^\infty(\T) \big)'.$  This is illustrated with the diagram below:
\begin{equation}
\label{diagrama}
\begin{array}{ccc}
\big( \mathcal{S}(\Z) \big)' & \stackrel{L}{\longrightarrow} & \big( \mathcal{S}(\Z) \big)' \\[12pt]
\mathcal{F}_\Z\!\downarrow\phantom{\!\mathcal{F}} & & \phantom{\mathcal{F}\!}\downarrow\!\mathcal{F}_\Z \\[6pt]
\big( C^\infty(\T) \big)' & \stackrel{L}{\longrightarrow} & \big( C^\infty(\T) \big)'
\end{array}
\end{equation}

Now we are in position to define properly the Riesz transforms.

\begin{prop}
\label{prop:Riesz-SZ}
The discrete Riesz transforms $\mathcal{R}$ and $\widetilde{\mathcal{R}}$ defined in \eqref{nuevasriesz} are operators acting on $\mathcal{S}(\Z)$ by the convolution with the sequences $\big \{ \frac1{\pi(n+1/2)}\big \}_{n\in \Z}$ and $\big \{ \frac1{\pi(n-1/2)} \big \}_{n\in\Z}$, respectively.
\end{prop}

\begin{proof}
We shall check that the operators we need to define the Riesz transforms can be seen as operators acting on~$C^\infty(\T)$.

Recall from Section~\ref{sec:disheatsemigroup} that the heat semigroup $W_t$ is given by convolution with the sequence $\{G(m,t)\}_{m\in \Z} = \{e^{-2t} I_m(2t)\}_{m\in \Z}$, for each $t$. We also showed that there exists a constant $C$ such that $G(m,t) \le \frac{C}{|m|+1},$ see  Proposition~\ref{prop:heat-kernel-estimates}. Hence
\begin{align*}
\big| \big \langle W_tf, g\big\rangle_\Z \big|
&= \big| \big\langle G(\cdot,t), \tilde{f}*g\big\rangle_\Z \big|
= \big|\sum_{n\in \Z} G(n,t) \tilde{f}*g(n) \big|\\
&\le C\sum_{n\in\Z} \frac1{|n|+1} |\tilde{f}*g(n)|.
\end{align*}
This inequality guaranties the boundedness of $W_t$ from $\mathcal{S}(\Z)$ into $\big( \mathcal{S}(\Z) \big)'$.  As a consequence of the computations made in the proof of Proposition~\ref{prop:heat-is-semigroup} (iii) we can understand the heat semigroup as the operator of multiplication by $e^{-4t\sin^2 \frac{(\cdot)}{2}}$, which is bounded from $C^\infty(\T)$ into $\big( C^\infty(\T)\big)'$. Even more, since $e^{-4t\sin^2 \frac{(\cdot)}{2}} \in C^\infty(\T)$, it is bounded from $C^\infty(\T)$ into itself.

Let us continue with the fractional integral
$(-\Delta_{\dis})^{-\alpha}$ defined in \eqref{eq:fractional-integral} for $0< \alpha < 1/2$. In our case, the action of the operator  $e^{t\Delta_{\dis}}$ is defined by $e^{t\Delta_{\dis}}(\varphi)(\theta)= e^{-4t\sin^2 \frac{\theta}{2}} \varphi(\theta)$. Hence
\begin{equation*}
(-\Delta_{\dis})^{-\alpha}\varphi(\theta) = \frac1{\Gamma(\alpha)} \int_0^\infty  e^{-4t\sin^2 \frac{\theta}{2}} \varphi(\theta) t^{\alpha}\,\frac{dt}{t}
= \Big(4\sin^2\frac{\theta}{2}\Big)^{-\alpha} \varphi(\theta).
\end{equation*}
Since the function $(4\sin^2\frac{\theta}{2})^{-\alpha}$ is integrable (remember that $0<\alpha < 1/2$) then $(-\Delta_{\dis})^{-\alpha}$ is bounded from
$C^\infty(\T)$ into  $C^\infty(\T).$

The last operator we are interested in is the first difference operator $D$ defined in \eqref{factorizacion}, acting on $\mathcal{S}(\Z)$. Observe that
\begin{align*}
\mathcal{F}_{\Z}(Df)(\theta) = \sum_{n\in\Z} f(n+1) e^{in\theta} -\sum_{n\in\Z} f(n) e^{in\theta}  = (e^{-i\theta}-1) \mathcal{F}_{\Z}(f)(\theta).
\end{align*}
That is, the operator  $\mathcal{F}_{\Z}\circ D \circ \mathcal{F}_{\Z}^{-1}$ is given by the multiplier $(e^{-i\theta}-1)$.

As a consequence, the operator  $\mathcal{R}^\alpha = \mathcal{F}_{\Z}\circ D (-\Delta_{\dis})^{-\alpha} \circ \mathcal{F}^{-1}_{\Z}$, $0< \alpha < 1/2$, (bounded from
$C^\infty(\T)$ into itself) is associated with the multiplier
\begin{align*}
(e^{-i\theta}-1)\Big(4\sin^2\frac{\theta}{2}\Big)^{-\alpha} &=
(e^{-i\theta}-1)\Big(\Big|2\sin \frac{\theta}{2}\Big|^2 \Big)^{-\alpha}
= e^{-i\theta/2} \frac{(e^{-i\theta/2}-e^{i\theta/2})}
{2^{2\alpha} |\sin\theta/2|^{2\alpha}}\\
&= e^{-i\theta/2} \frac{-2i\, \sin\theta/2}{2^{2\alpha} |\sin\theta/2|^{2\alpha}} .
\end{align*}
Hence  for $0\le \theta \le 2\pi,$
we have
$\lim_{\alpha \to (1/2)^{-}} (e^{-i\theta}-1)\Big(4\sin^2\frac{\theta}{2}\Big)^{-\alpha} = -ie^{-i\theta/2}$.
Therefore $\mathcal{R} = \lim_{\alpha \to (1/2)^{-}} \mathcal{R}^\alpha$, defined as the operator of multiplication by the function $-ie^{-i\theta/2}$,
is an operator bounded from $C^\infty(\T)$ into  $C^\infty(\T).$
On the other hand,
\begin{align*}
\frac1{2\pi} \int_{0}^{2\pi} -ie^{-i \theta/2} e^{-in\theta} \,d\theta
&= -\frac1{2\pi} \int_{0}^{2\pi } ie^{-i(n+\frac12)\theta} \,d\theta
= \frac1{\pi(n+\frac12)} .
\end{align*}
By the diagram~\eqref{diagrama} above, we conclude that the operator  $\mathcal{R}$ is an operator acting on $\mathcal{S}(\Z)$ by the convolution with the sequence $\big \{ \frac1{\pi(n+1/2)} \big \}_{n\in \Z}$.

By using an analogous reasoning with $\widetilde{D}$, it can be checked that $\widetilde{\mathcal{R}}$, defined on $C^{\infty}(\T)$ as the operator of multiplication by the function $-ie^{i\theta/2}$, corresponds with a convolution operator  with the sequence $\big \{ \frac1{\pi(n-1/2)} \big \}_{n\in \Z}$ acting on $\mathcal{S}(\Z)$.
 \end{proof}

Observe that Proposition~\ref{prop:Riesz-SZ} implies that  $\mathcal{R}$ and $\widetilde{\mathcal{R}}$ are bounded in $\ell^p$, $1 < p< \infty$, and they are well defined in~$\ell^1$. Moreover, since the kernels $\big \{ \frac1{\pi(n+1/2)} \big \}_{n\in\Z}$ and $\big \{ \frac1{\pi(n-1/2)} \big \}_{n\in\Z}$ obviously satisfy the Calder\'on--Zygmund estimates we have the following (well known) result.

\begin{coro}
\label{discreteriesz}
Let $w \in A_p$, $1\le p < \infty$. Then the operators $\mathcal{R}$ and $\widetilde{\mathcal{R}}$ are bounded from
$\ell^p(w)$ into itself and also from $\ell^1(w)$ into weak-$\ell^1(w)$.
\end{coro}

\section{Riesz transforms as limits of ``harmonic'' functions}
\label{sec:Cauchy}

Recall the conjugate harmonic operators
\[
Q_t f= \mathcal{R}  P_t f  \quad\text{and}\quad  \widetilde{Q}_tf = \widetilde{\mathcal{R}}P_t f.
\]
In this section, we shall prove Theorem~\ref{thm:Cauchy}. First, we show that these operators are well defined and satisfy several properties for good functions.

\begin{prop} Let $Q_t$ and $\widetilde{Q}_t$ be defined as above and $f$ be a compactly supported function.
\begin{itemize}
\item[(i)] The operators $Q_t$, $\widetilde{Q}_t$ and $P_t$ satisfy the Cauchy--Riemann type equations
\[
\begin{cases}
\partial_t(Q_tf) = -D(P_tf), \\
 \widetilde{D}(Q_tf) = \partial_t(P_tf);
\end{cases}
\quad \quad\begin{cases}
\partial_t(\widetilde{Q}_tf)
= -\widetilde{D}(P_tf),\\
 D(\widetilde{Q}_tf) = \partial_t(P_tf).
\end{cases}
\]
Moreover, $\partial^2_{tt} Q_tf(n) + \Delta_{\dis} Q_tf(n) =0$
and $\partial^2_{tt} \widetilde{Q}_tf(n)  + \Delta_{\dis} \widetilde{Q}_tf(n) =0$.
\item[(ii)] We have
\[
\lim_{t\to 0} Q_tf(n) = \mathcal{R} f(n) \quad \text{and}\quad \lim_{t\to 0} \widetilde{Q}_tf(n) = \widetilde{\mathcal{R}} f(n),
\]
for $n \in \mathbb{Z}$.

\item[(iii)] We have $\big| \mathcal{F}_{\Z}(Q_t f)(\theta) \big | + \big| \mathcal{F}_{\Z}(\widetilde{Q}_t f)(\theta) \big | \le C | \mathcal{F}_{\Z}(f)(\theta)|$.
\end{itemize}
\end{prop}
\begin{proof}
By the results on Section~\ref{sec:disheatsemigroup} and Section~\ref{sec:Riesz}, it is clear that the operators $Q_t$ and $\widetilde{Q}_t$ can be defined as operators of multiplication by $-ie^{-i\theta/2}e^{-2t|\sin \theta/2|}$ and $-ie^{i\theta/2}e^{-2t|\sin \theta/2|}$, respectively. Then, the conclusions are obvious.
\end{proof}

\begin{rem}
By conjugate \textit{harmonic} functions we mean that $Q_tf(n)$ and $\widetilde{Q}_tf(n)$ are the \textit{harmonic} conjugate functions of the \textit{harmonic} function $P_tf(n)$ in the variables $(t,n)$, see Remark~\ref{Poisson} in Section~\ref{sec:disheatsemigroup}.
\end{rem}

\begin{proof}[Proof of Theorem~\ref{thm:Cauchy}]
We shall prove the theorem for $Q_t$ only, since the proof for $\widetilde{Q}_t$ is similar.
In Section~\ref{sec:Riesz} we saw that the ``natural'' Riesz transforms associated with the operator $\Delta_{\dis}$  are bounded on the spaces $\ell^p(w)$, see Corollary~\ref{discreteriesz}.
This together with Theorem~\ref{thm:semigroups-gg} show that the conjugate harmonic functions are well defined for all  functions in $\ell^p(w)$, $1 \le p < \infty$, and also that
$\lim_{t\to 0} Q_tf = \mathcal{R}f $ in $\ell^p(w)$ sense and pointwise for  $1 < p< \infty.$ In order to get the appropriate results in $\ell^1(w)$ we need to work a bit more. We will prove that the operator $\sup_{t\ge0} |Q_tf|$ can be viewed as the norm of vector-valued Calder\'on--Zygmund operators whose kernels satisfy standard estimates. Once this fact has been shown, (i) will be a direct consequence.

If we take the derivative
with respect to $\beta$ in the subordination formula~\eqref{subordination}, we have
\begin{equation}
\label{eq:subPt}
\frac1{\beta} e^{-\beta t} = \frac1{\sqrt{\pi}}
\int_0^\infty \frac{e^{-t^2/(4v)}}{\sqrt{v}} e^{-v\beta^2} \,dv.
\end{equation}
Then, observe that, for $\theta \in [0,2\pi]$, we obtain
\begin{align}
\label{astucia}
\nonumber
\mathcal{F}_{\Z}(\mathcal{R} P_tf)(\theta)
&= -ie^{-i\theta/2} e^{-t|2\sin \theta/2|}  \mathcal{F}_{\Z}(f)(\theta)\\
& \nonumber =
e^{-i\theta/2} \frac{2i \sin\theta/2}{2|\sin\theta/2|} e^{-t|2\sin \theta/2|} \mathcal{F}_{\Z}(f)(\theta) \\
&= (e^{-i\theta}-1) \frac{1} {(4\sin^2\theta/2)^{1/2} } e^{-t|2\sin \theta/2|} \mathcal{F}_{\Z}(f)(\theta) \\
&= \nonumber
\Big( \frac1{\sqrt{\pi}} \int_0^\infty e^{-t^2/(4v)} (e^{-i\theta}-1)
e^{-4v\sin^2(\theta/2)}\frac{dv}{v^{1/2}}\Big)  \mathcal{F}_{\Z}(f)(\theta),
\end{align}
where in the last identity we used~\eqref{eq:subPt}.
Formula~\eqref{astucia} allows us to write, for functions $f\in S(\Z),$
\begin{equation}
\label{subordinationQ}
\begin{aligned}
\{ Q_tf(n) \}_{t\ge0} &= \{ D L^{-1/2} P_t f(n) \}_{t\ge0}  \\
&= \Big\{ \frac1{\sqrt{\pi}} \int_0^\infty e^{-t^2/(4v)} DW_vf(n) \frac{dv}{v^{1/2}} \Big\}_{t\ge0}.
\end{aligned}
\end{equation}

Now we shall see that the kernel associated with the operator~\eqref{subordinationQ} satisfies the Calder\'on--Zygmund estimates; we use $Q(m,t)$ to denote this kernel. We consider only $m>0$, as in previous cases. By reproducing the arguments given in the proof of Proposition~\ref{prop:heat-kernel-estimates},
we get
\begin{align*}
|Q(m,t)| &= \frac{4^m\,2}{\pi \Gamma(m+1/2)} \\*
&\qquad \times \int_0^\infty e^{-t^2/(4v)}
\int_0^v e^{-4w} w^{m-1} \Big( \frac{w}{v}\Big)^{3/2}
\Big(1-\frac{w}{v}\Big)^{m-1/2}\,dw\, \frac{dv}{v^{1/2}}.
\end{align*}
Observe that
\begin{align*}
&\int_0^\infty e^{-t^2/(4v)}
\int_0^v e^{-4w} w^{m-1} \Big(\frac{w}{v}\Big)^{3/2}
\Big(1-\frac{w}{v}\Big)^{m-1/2}\, dw \,\frac{dv}{v^{1/2}} \\
&\qquad\quad =
\int_0^\infty e^{-t^2/(4v)} \int_0^1 e^{-4vs} s^{m-1} s^{3/2}
{(1-s)}^{m-1/2}\,ds\, v^{m-1/2}\,dv \\
&\qquad\quad =
\int_0^1 s^{m-1} s^{3/2} {(1-s)}^{m-1/2}
\int_0^\infty e^{-t^2/(4v)} e^{-4vs} v^{m-1/2} \,dv \, ds \\
&\qquad\quad =
\int_0^1 s^{m-1} s^{3/2} {(1-s)}^{m-1/2}
\int_0^\infty e^{-t^2s/r} e^{-r} \Big( \frac{r}{4s}\Big)^{m-1/2} \frac{dr}{4s}\, ds \\
&\qquad\quad \le
\int_0^1\frac{ s^{m-1}s^{3/2} {(1-s)}^{m-1/2}}{(4s)^{m-1/2} s}
\int_0^\infty e^{-r} r^{m-1/2} \,dr \,ds \\
&\qquad\quad = C 4^{-m} \int_0^1 {(1-s)}^{m-1/2} ds \, \Gamma(m+1/2) \\
&\qquad\quad = C 4^{-m} \, \frac{\Gamma(1) \Gamma(m+1/2)}{\Gamma(m+3/2)} \Gamma(m+1/2)
\sim C 4^{-m} \, \frac1{m}  \Gamma(m+1/2).
\end{align*}
Hence we have proved $\sup_{t\ge0} |Q(m,t)| \le \frac{C}{|m|+1}$, where $C$ is a constant independent of~$m$.

To end the proof of the theorem we need to prove the smoothness of the kernel, that is
\[
\sup_{t\ge0} |Q(m+1,t)-Q(m,t)| \le \frac{C}{m^2+1},
\]
with $C$ a constant independent of~$m$.
By using~\eqref{eq:difIs3terms}, we have
\begin{align*}
& |DQ(m,t)| = \bigg|\int_0^{\infty} e^{-\frac{t^2}{4v}}e^{-2v}
(I_{m+2}(2v)-2I_{m+1}(2v)+I_m(2v))
\frac{dv}{v^{1/2}}\bigg| \\
&\quad =\frac{1}{\sqrt{\pi} \,\Gamma(m+1/2)}\bigg|\int_0^{\infty}e^{-\frac{t^2}{4v}}v^m
\bigg(\frac{1}{v}
\int_{-1}^1e^{-2v(1+s)}s(1-s^2)^{m-1/2}\,ds \\
&\qquad\quad +\int_{-1}^1e^{-2v(1+s)}(1+s)^2(1-s^2)^{m-1/2}\,ds\bigg)\,dv\bigg|\\
&\quad\le \frac{4^m}{\sqrt{\pi}\,\Gamma(m+1/2)}\bigg(\int_0^1u^{m-1/2}(1-u)^{m-1/2}
\int_0^{\infty} e^{-\frac{t^2}{4v}}v^{m-3/2}e^{-4vu}\,dv\,du\\
& \qquad\quad +\int_0^1u^{m+3/2}(1-u)^{m-1/2}
\int_0^{\infty} e^{-\frac{t^2}{4v}}v^{m-1/2}e^{-4vu}\,dv\,du\bigg)=:I_1+I_2.
\end{align*}
For each one, we proceed similarly as in the growth estimates and we get
\[
I_1 \le \frac{C}{\sqrt{\pi}\,\Gamma(m+1/2)}
\frac{\Gamma(1)\Gamma(m+1/2)}{\Gamma(m+3/2)}
\Gamma(m-1/2) \sim (m+1/2)^{-2}
\]
and
\[
I_2\le \frac{C}{\sqrt{\pi}\,\Gamma(m+1/2)}
\frac{\Gamma(2)\Gamma(m+1/2)}{\Gamma(m+5/2)}
\Gamma(m+1/2)\sim (m+1/2)^{-2}.
\]
This ends the proof of~(i).

In order to prove (ii) and (iii) we observe that $Q_tf$, for $f \in \mathcal{S}(\Z)$, can be alternatively defined as
\[
Q_tf= \int_0^t DP_sf \, ds - \mathcal{R}f.
\]
By applying Fourier transform, we see that this last definition is valid for any function on $\ell^p(w)$, $1\le p < \infty$, $w\in A_p$. Moreover it can be checked that $Q_t$ satisfies (ii) and~(iii).
\end{proof}

\section{Technical results about the functions $I_k$}
\label{sec:preliminaries}

Let $I_k$ be the modified Bessel function of the first kind and order $k\in \Z$, defined as
\begin{equation}
\label{eq:Ik}
  I_k(t) = i^{-k} J_k(it) = \sum_{m=0}^{\infty} \frac{1}{m!\,\Gamma(m+k+1)} \left(\frac{t}{2}\right)^{2m+k}.
\end{equation}
Since $k$ is an integer (and $1/\Gamma(n)$ is taken to equal zero if $n=0,-1,-2,\ldots$), the function $I_k$ is defined in the whole real line (even in the whole complex plane, where $I_k$ is an entire function). We list several properties of~$I_k$. Most of them can be found in \cite[Chapter~5]{Lebedev} and~\cite{OlMax}.

It is verified that
\begin{equation}
\label{eq:simetriak}
I_{-k}(t) = I_k(t)
\end{equation}
for each $k \in \Z$. Besides, from~\eqref{eq:Ik} it is clear that
\begin{equation}
\label{eq:Ik0}
I_0(0) = 1 \quad \text{ and } \quad I_k(0) = 0 \quad \text{ for } \quad k \ne 0.
\end{equation}
The following identity is called \textit{Neumann's identity}, see \cite[Chapter~II, formula~(7.10)]{Feller}:
\begin{equation}
\label{eq:Neumann}
  I_{r}(t_1+t_2) = \sum_{k \in \Z} I_k(t_1) I_{r-k}(t_2) \quad \text{ for } \quad r \in \Z;
\end{equation}
this formula is an easy consequence of the generating function
\[
  e^{\frac12 t (u+u^{-1})} = \sum_{k \in \Z} u^k I_{k}(t)
\]
that sometimes serves as definition for~$I_k$ (see, for instance, \cite[formula 10.35.1]{OlMax}).
Other properties of $I_k$ are
\begin{equation}
\label{eq:Ik>0}
  I_k(t) \ge 0
\end{equation}
for every $k \in \Z$ and $t\ge0$, and
\begin{equation}
\label{eq:sumIk}
  \sum_{k \in \Z} e^{-2t} I_k(2t) = 1.
\end{equation}

Clearly, from \eqref{eq:Ik}, there exist constants $C$, $c>0$ such that
\begin{equation}
\label{eq:asymptotics-zero}
c t^k\le I_k(t)\le C t^k \quad \text{ for } t\to0^+.
\end{equation}
Moreover, it is well known (see \cite{Lebedev}) that
\begin{equation}
\label{eq:asymptotics-infinite}
I_k(t)=C e^t t^{-1/2}+ R_k(t),
\end{equation}
where
\[
|R_k(t)|\le C_k e^tt^{-3/2}, \quad \text{ for } t\to\infty.
\]

The modified Bessel function $I_k(t)$ also satisfies
\[
  \frac{\partial}{\partial t} I_k(t) = \frac{1}{2} (I_{k+1}(t) + I_{k-1}(t)),
\]
and from this it follows immediately
\begin{equation}
\label{eq:derivadaCalort}
  \frac{\partial}{\partial t} (e^{-2t} I_k(2t)) = e^{-2t} (I_{k+1}(2t) - 2I_{k}(2t) + I_{k-1}(2t)).
\end{equation}

The next identity is known as Schl\"afli's integral
representation of Poisson type for modified Bessel functions (see \cite[(5.10.22)]{Lebedev}), and it is valid for a real number $\nu>-\frac12$:
\begin{equation}
\label{eq:Schlafli}
I_{\nu}(z) = \frac{z^{\nu}}{\sqrt{\pi}\,2^{\nu}\Gamma(\nu+1/2)}
\int_{-1}^1e^{-zs} (1-s^2)^{\nu-1/2}\,ds, \quad
|\arg z|<\pi, \quad \nu>-\frac12.
\end{equation}
If we integrate by parts once, twice and three times in~\eqref{eq:Schlafli}, we get, respectively,
\begin{equation}
\label{eq:SchlafliII}
I_{\nu}(z) = -\frac{z^{\nu-1}}{\sqrt{\pi}\,2^{\nu-1}\Gamma(\nu-1/2)}
\int_{-1}^1e^{-zs} s(1-s^2)^{\nu-3/2}\,ds, \quad \nu>\frac12,
\end{equation}
\begin{equation}
\label{eq:SchlafliIII}
I_{\nu}(z)=\frac{z^{\nu-2}}{\sqrt{\pi}\,2^{\nu-2}\Gamma(\nu-3/2)}
\int_{-1}^1e^{-zs} \frac{1+zs}{z}s(1-s^2)^{\nu-5/2}\,ds, \quad \nu>\frac32,
\end{equation}
and
\begin{align}
\label{eq:SchlafliIV}
I_{\nu}(z) &= -\frac{z^{\nu-3}}{\sqrt{\pi}\,2^{\nu-3}\Gamma(\nu-5/2)} \\*
\notag& \qquad \times \int_{-1}^1 e^{-zs} \,\frac{s(s^2z^2+3sz+3)}{z^2}(1-s^2)^{\nu-7/2}\,ds,
\quad \nu>\frac52.
\end{align}
Combining~\eqref{eq:Schlafli} and~\eqref{eq:SchlafliII} we get, for $\nu>-1/2$,
\begin{equation}
\label{eq:difIs}
I_{\nu+1}(z)-I_\nu(z)
= -\frac{z^{\nu}}{\sqrt{\pi}\,2^{\nu}\Gamma(\nu+1/2)}  \int_{-1}^1e^{-zs} (1+s)(1-s^2)^{\nu-1/2}\,ds.
\end{equation}
Combining~\eqref{eq:Schlafli}, \eqref{eq:SchlafliII} and~\eqref{eq:SchlafliIII} we obtain, for $\nu>-1/2$,
\begin{equation}
\label{eq:difIs3terms}
\begin{aligned}
& I_{\nu+2}(z)-2I_{\nu+1}(z)+I_{\nu}(z)
= \frac{z^{\nu}}{\sqrt{\pi}\,2^{\nu}\Gamma(\nu+1/2)} \\*
& \qquad\times \Big(\frac{2}{z}
\int_{-1}^1e^{-zs} s(1-s^2)^{\nu-1/2}\,ds
+ \int_{-1}^1e^{-zs} (1+s)^2(1-s^2)^{\nu-1/2}\,ds \Big).
\end{aligned}
\end{equation}
Combining~\eqref{eq:Schlafli}, \eqref{eq:SchlafliII}, \eqref{eq:SchlafliIII} and~\eqref{eq:SchlafliIV} we obtain, for $\nu>-1/2$,
\begin{equation}
\label{eq:difIs4terms}
\begin{aligned}
&I_{\nu+3}(z)-3I_{\nu+2}(z)+3I_{\nu+1}(z)-I_\nu(z)
= \frac{z^{\nu}}{\sqrt{\pi}\,2^{\nu} \Gamma(\nu+1/2)} \\*
&\quad\times \Big(\frac{3}{z^2}
\int_{-1}^1e^{-zs} s(1-s^2)^{\nu-1/2}\,ds
+ \frac{3}{z}\int_{-1}^1e^{-zs} s(1+s)(1-s^2)^{\nu-1/2}\,ds \\
&\qquad + \int_{-1}^1e^{-zs} (1+s)^3(1-s^2)^{\nu-1/2}\,ds \Big).
\end{aligned}
\end{equation}


\end{document}